\documentclass[12pt]{amsart}
\usepackage{amsfonts, amstext}
\usepackage[usenames,dvipsnames]{xcolor}	
\usepackage{amsmath, amsthm, amssymb, amsfonts}						
\usepackage{fullpage}
\usepackage{enumitem}
\usepackage{mathrsfs}
\usepackage{tikz-cd}
\usepackage{hhline}
\usepackage{textcomp}
\usepackage{mathbbol}

\usepackage{hyperref}

\newcommand{\mcal}{\mathcal}
\newcommand{\mscr}{\mathscr}
\newcommand*{\sHom}{\mscr{H}\kern -.5pt om}
\newcommand*{\sExt}{\mscr{E}\kern -.5pt xt}

\newcommand{\PP}{\mathbb{P}}

\newcommand{\C}{\mathbb{C}}

\renewcommand{\O}{\mcal{O}}

\DeclareMathOperator{\Ext}{Ext}

\DeclareMathOperator{\coker}{coker\,}

\DeclareMathOperator{\codim}{codim\,}

\DeclareMathOperator{\Sing}{Sing}

\DeclareMathOperator{\Sym}{Sym}

\DeclareMathOperator{\rk}{rk}
\DeclareMathOperator{\adj}{adj}
\DeclareMathOperator{\sing}{sing}
\DeclareMathOperator{\PGL}{PGL}
\DeclareMathOperator{\Bl}{Bl}
\DeclareMathOperator{\Hilb}{Hilb}

\newtheorem{theorem}{Theorem}[section]

\newtheorem{proposition}[theorem]{Proposition}

\newtheorem{lemma}[theorem]{Lemma}

\theoremstyle{definition}
\newtheorem{example}[theorem]{Example}
\newtheorem{definition}[theorem]{Definition}

\theoremstyle{remark}
\newtheorem{remark}[theorem]{Remark}

\title{The Harder-Narasimhan Filtration of the Normal Bundle of a Trigonal Canonical Curve}
\author{Henry Fontana}
\address{Department of Mathematics, Statistics and Computer Science, University of Illinois at Chicago, Chicago, IL, 60607}
\email{hfonta2@uic.edu}

\begin{document}

\begin{abstract}
A trigonal canonical curve $C$ lies on a rational normal surface scroll $S$ in $\PP^{g-1}$. In this note we use this fact to compute the Harder-Narasimhan Filtration of the normal bundle of a general such $C$ in $\PP^{g-1}$. We also compute the Harder-Narasimhan filtration of the Normal bundle of a general canonical curve of genus $6$.
\end{abstract}

\maketitle

\section{Introduction}
Let $C$ be a smooth, irreducible, non-hyperelliptic curve over an algebraically closed field $k$. There is a canonical embedding $\phi_K:C \hookrightarrow \PP^{g-1}$ which reflects the intrinsic properties of $C$. The normal bundle $N_{C/\PP^{g-1}}$ controls the deformations of $C$ in this embedding, and it is therefore useful to understand the structure of $N_{C/\PP^{g-1}}$. It was conjectured by Aprodu, Farkas, and Ortega in \cite{AFO16} that $N_{C/\PP^{g-1}}$ is semi-stable for the general canonical curve $C$ once the genus $g$ is large enough. This conjecture was confirmed by Coşkun, Larson, and Vogt in \cite{CLV23} where they proved that if $g \notin \{4,6\}$ then $N_{C/\PP^{g-1}}$ is semi-stable for a general canonical curve of genus $g$.

The result of \cite{CLV23} raises the question of which special curves in the non-hyperelliptic locus of $\mathcal{M}_g$ have canonical models such that $N_{C/\PP^{g-1}}$ is unstable. Furthermore in the case of instability we can ask for the Harder-Narasimhan filtration of $N_{C/\PP^{g-1}}$. For example we will show that $N_{C/\PP^{g-1}}$ is unstable when $C$ is trigonal, due to the fact that $C$ lies on a surface scroll $S \subset \PP^{g-1}$. The main result of this note is the following Theorem computing the HN-filtration of $N_{C/\PP^{g-1}}$.

\begin{theorem} \label{MT}
Let $C$ be a general trigonal canonical curve of genus $g$ embedded in $\PP^{g-1}$ and let $S$ be the rational normal scroll containing $C$. Then
$$N_{C/S} \subset N_{C/\PP^{g-1}}$$
is the Harder-Narasimhan filtration of $N_{C/\PP^{g-1}}$.
\end{theorem}

We already know from \cite{CLV23} that canonical curves with $N_{C/\PP^{g-1}}$ unstable are rare. It is expected that $N_{C/\PP^{g-1}}$ will only be unstable if it is forced to be by the geometry of $C$. A future goal would be to describe all the geometric conditions which lead to instability of $N_{C/\PP^{g-1}}$. For example, it remains to compute the HN-filtration when $C$ is a genus $g$ curve of gonality $4$.

\begin{itemize}
\item In section $2$ we will recall some preliminary results such as the normal bundle of a rational normal curve and the definition of semi-stability on connected nodal curves.
\item Section $3$ is devoted to showing that $N_{C/\PP^{g-1}}$ has a destabilizing subbundle in the trigonal case.
\item In section $4$ we prove Theorem \ref{MT} by degenerating $C$ to a union of rational curves. 
\item In the final section we discuss semi-stability of the normal bundle of tetragonal canonical curves. The main result of this section is the computation of the HN-filtration of $N_{C/\PP^{5}}$ where $C$ is a canonical curve of genus $6$ (recall that all genus $6$ curves are tetragonal).
\end{itemize}

A future problem is to determine the HN-filtration of $N_{C/\PP^{g-1}}$ for tetragonal canonical curves of genus $g \geq 7$. The expectation is that the HN-filtration should be $N_{C/Q} \subset N_{C/\PP^{g-1}}$ where $Q \subset \PP^{g-1}$ is the threefold scroll containing $C$. Also our work in the final section reveals a potential strategy for determining the stability of the normal bundle of a threefold scroll $N_{Q/\PP^{g-1}}$. In the case $g = 6$ case we are able to find a rational curve $C \subset Q$ such that $N_{Q/\PP^5} \rvert_C$ is semi-stable which implies $N_{Q/\PP^{5}}$ must be semi-stable.

\subsection{Acknowledgments}
Thank you to my advisor İzzet Coşkun for introducing me to the problem and guiding me as I worked on the solution. I would also like to thank Eric Larson, Isabel Vogt, and Sebastian Casalaina-Martin for valuable input and discussions.

\section{Preliminaries}\label{sec:Prelims}

We will follow the conventions and definitions established in \cite{ACGH}. Given any smooth algebraic curve $C$ over an algebraically closed field $k$ there is a finite morphism $\phi:C \to \PP^1$. The minimal degree of a morphism $C \to \PP^1$ is the \textbf{gonality} of $C$. Curves of genus $g \geq 2$ and gonality $2$ are hyperelliptic and curves with gonality $3$ are \textbf{trigonal}. A smooth curve of genus $g$ is non-hyperelliptic iff the canonical linear series gives an embedding $C \to \PP^{g-1}$ and the image of such an embedding is a \textbf{canonical model} of $C$. We will follow the definition of \cite{FS91} and refer to a curve $C \subset \PP^{g-1}$ as a \textbf{canonical curve} if $\O_{\PP^1}(1) \cong \omega_C$, $h^0(\O_C)=1$, and $h^0(\omega_C)=g$. Canonical curves form an irreducible component in the Hilbert scheme of genus $g$, degree $2g-2$ curves in $\PP^{g-1}$ and by \textbf{general canonical curve} we mean an element lying in some Zariski open subset of this component.

We will use the term bundle to refer to an algebraic vector bundle over $k$. We can associate to any bundle $\mathcal{E}$ its rank and degree. The \textbf{slope} of $\mathcal{E}$ is defined to be
$$\mu(\mathcal{E})=\frac{\deg(\mathcal{E})}{\rk(\mathcal{E})}$$
We say that a bundle $\mathcal{E}$ is \textbf{slope semi-stable} if $\mu(\mathcal{F}) \leq \mu(\mathcal{E})$ for all proper subbundles $\mathcal{F} \subset \mathcal{E}$. If the inequality is strict for all proper subbundles then $\mathcal{E}$ is \textbf{slope stable}. The following theorem, which can be found in section $5.4$ of \cite{LePotier}, shows that every vector bundle can be built up from semi-stable bundles by taking successive extensions. The filtration in the Theorem is called the \textbf{Harder-Narasimhan} filtration of $\mathcal{E}$.

\begin{theorem} \label{T1}
Let $\mathcal{E}$ be a vector bundle on a complex projective variety $X$. There is a filtration
$$0 =\mathcal{E}_0 \subset \mathcal{E}_1 \subset \mathcal{E}_2 \subset \cdots \subset \mathcal{E}_k=\mathcal{E}$$
such that if $\mathcal{F}_i=\mathcal{E}_i / \mathcal{E}_{i-1}$ for $1 \leq i \leq k$ then
\begin{enumerate}
    \item Each $\mathcal{F}_i$ is semi-stable 
    \item $\mu(\mathcal{F}_{i-1}) > \mu(\mathcal{F}_i)$ for $i=1, \cdots, k$
\end{enumerate}
\end{theorem}

In order to detect semi-stability we need to know the slope of $N_{C/\PP^{g-1}}$. We will now compute the degree, and therefore the slope, of the normal bundle $C$ of a nonsingular curve of degree $d$ embedded in $\PP^n$ for some $n \geq 1$. In what follows we will write $N_C$ in place of $N_{C/\PP^n}$. We compute
$$\deg(N_C)=\deg(T_{\PP^n} \rvert_C) - \deg(T_C)$$
$$\rk(N_C)=\rk(T_{\PP^n} \rvert_C) - \rk(T_C)$$
The degree of $T_{\PP^n} \rvert_C$ is determined by the degree of $C$ and the Euler sequence
\[\begin{tikzcd}
	0 & {\O_{\PP^n}} & {\O_{\PP^n}(1)^{\oplus n+1}} & {T_{\PP^n}} & 0
	\arrow[from=1-1, to=1-2]
	\arrow[from=1-4, to=1-5]
	\arrow[from=1-2, to=1-3]
	\arrow[from=1-3, to=1-4]
\end{tikzcd}\]
which together imply $\deg(T_{\PP^n} \rvert_C)=d(n+1)$. Therefore we have $\deg(N_{C})=d(n+1)+2g-2$ so that
$$\mu(N_C)=\frac{d(n+1)+2g-2}{n-1}$$
In particular we deduce the slope of $N_{C/\PP^{g-1}}$ where $C \subset \PP^{g-1}$ is a smooth canonical curve of genus $g$.

\begin{proposition} \label{P1}
If $C \subset \PP^{g-1}$ is a smooth canonical curve of genus $g$ then
$$\mu(N_{C/\PP^{g-1}})=2g+4+\frac{6}{g-2}$$
\end{proposition}

Recall that given fixed integers $a,r$ the only semi-stable bundle on $\PP^1$ of slope $a$ and rank $r$ has the form
$${\O_{\PP^1}\left(a\right)}^{\oplus r}$$ 
Since our strategy of proof is to degenerate to a union of rational curves we will constantly be using the well known fact, see for example \cite{CR19}, that the normal bundle of a rational normal curve $C$ of degree $d$ in $\PP^d$ is a semi-stable bundle on $\PP^1$ of slope $d+2$ and rank $d-1$.

\begin{lemma} \label{L1}
If $C \subset \PP^d$ is a rational normal curve of degree $d$ then
$$N_{C/\PP^d} \cong \O_{\PP^1}(d+2)^{\oplus d-1}$$
\end{lemma}

Next we must briefly recall the construction of the \textbf{pointing bundle} $N_{C \to p}$. For a much more complete description of pointing bundles the reader should consult section $5$ of \cite{ALY19}. Given a smooth curve $C \subset \PP^r$ and a point $p \in \PP^r$ let $\pi_p:C \to \PP^{r-1}$ be the restriction to $C$ of projection from $p$ onto a hyperplane $\PP^{r-1}$. Furthermore let $U \subset C$ be the open set consisting of points $q$ such that the projective tangent space $\mathbb{T}_q(C) \subset \PP^{r}$ does not contain $p$. Then on the open set $U$ the following is an exact sequence of vector bundles

\[\begin{tikzcd}
	0 & {\mathcal{L}} & {N_{C/\PP^{r-1}} \rvert_U} & {\pi_p^* N_{\pi_p(C)/\PP^{r-1}} \rvert_U} & 0
	\arrow[from=1-1, to=1-2]
	\arrow[from=1-2, to=1-3]
	\arrow[from=1-3, to=1-4]
	\arrow[from=1-4, to=1-5]
\end{tikzcd}\]
Geometrically the fibers of the kernel $\mathcal{L}$ are the normal directions in $N_{C/\PP^{r-1}}$ pointing towards $p$, which of course only makes sense for fibers over $q \in U$ at which the direction towards $p$ is not tangent to $C$. Furthermore by the curve-to-projective extension theorem the line bundle $\mathcal{L}$ is the restriction of a subbundle $N_{C \to p} \subset N_{C/\PP^r}$ which we call the pointing bundle towards $p$.

\begin{example} \label{E1}
Let $C \subset \PP^d$ be a rational normal curve of degree $d$ and choose a point $p \in C$. Projection from $p$ defines a birational map from $C$ to a rational normal curve $X$ of degree $d-1$ in $\PP^{d-1} \subset \PP^d$. The cone over $X$ with vertex $p$ is a surface $Q_p \subset \PP^d$ which contains $C$ and is singular exactly at $p$. Suppose $\pi=\iota \circ \pi_p: \hat{Q_p} \to \PP^{d}$ is the projection map of the blowup $\pi_p:\hat{Q_p} \to Q_p$ of $Q_p$ at $p$ followed by inclusion $\iota:Q_p \to \PP^{d}$. Then the differential 
$$d\pi \rvert_C: T_{\hat{Q_p}} \rvert_C \to T_{\PP^d} \rvert_C$$
is an isomorphism along $T_C$ and drops rank precisely at $p \in C$. Let $E \subset \hat{Q_p}$ be the exceptional divisor of the blow up and observe that $E$ and $C$ intersect with multiplicity $1$ at $p$. Therefore, if we write out the map $\pi$ in local coordinates the determinant of $(d-1) \times (d-1)$ minors of the Jacobian matrix vanish to order $1$ at $p$. It follows that we get an inclusion of vector bundles
$$N_{C/\hat{Q_p}} \subset N_{C/\PP^{d}}(-p)$$
so that $N_{C/\hat{Q_p}}(p)$ is a subbundle of $N_{C/\PP^{d}}$. Furthermore from the construction it is clear that we have
$$N_{C/\hat{Q_p}}(p) \rvert_U = \ker(N_{C/\PP^d} \rvert_U \to \pi_p^* N_{\pi_p(C)/\PP^{d-1}} \rvert_U)$$
where $U=C \backslash \{p\}$, hence we conclude that $N_{C/\hat{Q_p}}(p)=N_{C \to p}$.
\end{example}

\begin{lemma} \label{L5}
If $C \subset \PP^d$ is a degree $d$ rational normal curve then given any $d-1$ distinct points $p_1, \dots, p_{d-1} \in C$ the induced map
$$N_{C \to p_1} \oplus N_{C \to p_2} \oplus \dots \oplus N_{C \to p_{d-1}} \to N_{C/\PP^{d}}$$
is an isomorphism.
\end{lemma}

\begin{proof}
Suppose that $p_1, \dots, p_{d-1}$ are distinct points on $C$ and $r \in C$ is any point with $r \neq p_i$ for all $i$. For any $i$ the image of the fiber of $N_{C \to p_i}$ over $r$ in the fiber of $N_{C/\PP^d}$ over $r$ is
$$\frac{T_{L_i,r} + T_{C,r}}{T_{C,r}}$$
where $L_i$ is the line from $p_i$ to $r$. Therefore the natural map
$$F:N_{C \to p_1} \oplus N_{C \to p_2} \oplus \dots \oplus N_{C \to p_{d-1}} \to N_{C/\PP^{d}}$$
is injective on the fiber over $r$ if the projective tangent space $\mathbb{T}_r(C)$ is not contained in the hyperplane $H$ spanned by the points $p_1, p_2, \dots, p_{d-1},r$. If $\mathbb{T}_r(C)$ were contained in this span then $H$ would intersect $C$ in at least $d+1$ points counted with multiplicity, this is a contradiction since $C$ has degree $d$. Hence we conclude that the map $F$ is injective as a morphism of sheaves, because it is injective away from a finite set of points. Note that the bundles $N_{C \to p_1} \oplus N_{C \to p_2} \oplus \dots \oplus N_{C \to p_{d-1}}$ and $N_{C/\PP^d}$ have the same rank and first Chern class by Lemma \ref{L1} and Example \ref{E1}. Therefore the cokernel has rank $0$ and first Chern class $0$ implying the map $F$ is surjective.
\end{proof}

We will need several results from \cite{CLV22}, in particular those regarding the adjusted slope of a vector bundle on a connected nodal curve. Let $X$ be a connected nodal curve and
$$\nu: \tilde{X} \to X$$
the normalization of $X$. For a node $p \in X$ the fiber $\nu^{-1}(p)$ consists of two points $p_1,p_2$. If we pullback a vector bundle $\mathcal{E}$ on $X$ to $\tilde{X}$ the fibers of $\mathcal{M}=\nu^{*} \mathcal{E}$ over $p_1$ and $p_2$ are both naturally identified with $\mathcal{E}_{p}$.
Hence given a subbundle $\mathcal{F} \subset \mathcal{M}$ we can consider $\mathcal{F}_{p_1} \cap \mathcal{F}_{p_2}$ as a subspace of $\mathcal{E}_{p}$. We will use the notation of \cite{CLV22} and write $\codim_{\mathcal{F}}(\mathcal{F}_{p_1} \cap \mathcal{F}_{p_2})$ for the codimension of $\mathcal{F}_{p_1} \cap \mathcal{F}_{p_2}$ in either $\mathcal{F}_{p_1}$ or $\mathcal{F}_{p_2}$ which are equal because $\dim(\mathcal{F}_{p_1})=\dim(\mathcal{F}_{p_2})$. The following definition of the \textbf{adjusted slope} of a subbundle $\mathcal{F} \subset \mathcal{M}$ can be found on page $3$ of \cite{CLV22}.

\begin{definition}
Let $X$ be a connected curve with only nodes as singularities. The adjusted slope of a subbundle $\mathcal{F} \subset \mathcal{M}=\nu^* \mathcal{E}$ is
$$\mu^{\adj}(\mathcal{F})=\mu(\mathcal{F}) - \frac{1}{\rk(\mathcal{F})} \sum_{p \in X_{\sing}} \codim_{\mathcal{F}}(\mathcal{F}_{p_1} \cap \mathcal{F}_{p_2})$$
\end{definition}

If $X$ is smooth then the adjusted slope reduces to the ordinary definition of slope for vector bundles on smooth curves. We say that a vector bundle $\mathcal{E}$ on a connected nodal curve is semi-stable if $\mu^{\adj}(\mathcal{F}) \leq \mu(\mathcal{E})$ for all proper subbundles $\mathcal{F} \subset \nu^* \mathcal{E}$. The following result from the preliminary section of \cite{CLV22} allows us to reduce the semi-stability of a bundle on a general curve to the semi-stability of the bundle on a specific connected nodal curve.

\begin{proposition} \label{P5}
Let $\mathscr{C} \to \Delta$ be a family of connected nodal curves over the spectrum of a discrete valuation ring and $\mathcal{E}$ a vector bundle on $\mathscr{C}$. If the special fiber $\mathcal{E} \rvert_0$ is semi-stable then the general fiber $\mathcal{E} \rvert_t$ is semi-stable.
\end{proposition}

In the final section on tetragonal curves we will need another Proposition from the \newline preliminary section of \cite{CLV22}.

\begin{proposition} \label{P6}
Let $\mathcal{E}$ be a vector bundle on a reducible nodal curve $X_1 \cup X_2$. If $\mathcal{E} \rvert_{X_1}$ and $\mathcal{E} \rvert_{X_2}$ are both semi-stable then $\mathcal{E}$ is semi-stable.
\end{proposition}

We will use Proposition \ref{P5} to prove Theorem \ref{MT} by letting $X=X_1 \cup X_2 \cup X_3$ for rational curves $X_i$ and then showing that $N_{S/\PP^{g-1}} \rvert_X$ is semi-stable with respect to the adjusted slope. To calculate the adjusted slope we need to be able to compute $N_{S/\PP^{g-1}} \rvert_{X_i}$ for each of the components $X_i$. This section ends with a series of lemmas that will allow us to compute this bundle for a few classes of curves on $S$. Note that for a general trigonal canonical curve we have $S \cong \PP^1 \times \PP^1$ if the genus is even and $S \cong \Bl_p \PP^2$ when the genus is odd.

\begin{lemma} \label{L2}
Suppose $Y \subset \PP^n$ is a subvariety of dimension $d > 0$ and let $H \cong \PP^{n-1}$ be a hyperplane meeting $Y$ transversely in a $(d-1)$-dimensional subvariety $X=Y \cap H$. Then
$$N_{X/\PP^{n-1}} \cong N_{Y/\PP^n} \rvert_X$$
\end{lemma}

\begin{proof}
We can factor the inclusion $X \subset \PP^n$ either as $X \subset Y \subset \PP^n$ or as $X \subset \PP^{n-1} \subset \PP^n$. We get a commutative diagram of the form
\[\begin{tikzcd}
	0 & {T_{X}} & {T_{Y} \rvert_X} & {\O_{X}(1)} & 0 \\
	0 & {T_{\PP^{n-1}} \rvert_{X}} & {T_{\PP^{n}} \rvert_X} & {\O_{X}(1)} & 0
	\arrow[from=1-1, to=1-2]
	\arrow[from=1-4, to=1-5]
	\arrow[from=1-3, to=1-4]
	\arrow[from=1-2, to=1-3]
	\arrow[from=2-1, to=2-2]
	\arrow[from=2-2, to=2-3]
	\arrow[from=2-3, to=2-4]
	\arrow[from=2-4, to=2-5]
	\arrow["\alpha", from=1-2, to=2-2]
	\arrow["\beta", from=1-3, to=2-3]
	\arrow["\gamma", from=1-4, to=2-4]
\end{tikzcd}\]
Since $Y$ meets $H$ transversely the image of $\beta$ is not contained in $T_{\PP^{n-1}} \rvert_{X}$. It follows that $\gamma$ is a nonzero morphism of line bundles so it must have rank $1$, i.e. it is an isomorphism. Then the Snake Lemma gives $\coker(\alpha) \cong \coker(\beta)$ as desired.
\end{proof}

\begin{lemma} \label{L3}
Let $W \subset \PP^n$ be a minimal degree nondegenerate surface scroll (i.e. a ruled surface over $\PP^1$) and $Y$ a rational normal curve of degree $k$ with $2 \leq k \leq n-1$. If there exists a linear space $\PP^k \subset \PP^n$ with $Y \subset W \cap \PP^k$ then $Y= W \cap \PP^k$.
\end{lemma}

\begin{proof}
We can write the class $[Y]=aE+mF$ for some $a,m$ where $E$ is a section of $W$ considered as a $\PP^1$ bundle and $F$ is a fiber. If $a=0$ then $Y$ is a disjoint union of fibers which is a contradiction. Suppose $a>1$ so that $F \cdot [Y]=a$, then the fibers of $W$ meet $Y$ in multiple points hence they meet $\PP^k$ in multiple points. But the fibers of $W$ are lines so this can only happen if all these lines are contained in $\PP^k$. This implies $W \subset \PP^k$ which contradicts the nondegenerate condition. We conclude $a=1$ i.e.
$$[Y]=E+mF$$
Now suppose $Y \cup \{p\} \subset W \cap \PP^k$ for some $p \notin Y$, we will argue towards a contradiction. Since $F\cdot [Y]=1$ the fiber $L$ of $W$ containing $p$ intersects $Y$ in another point $q \neq p$. We have $p,q \in \PP^k$ so that $L$ must be contained in $\PP^k$. Thus we see that $Y \cup L \subset W \cap \PP^k$. Since $W$ is nondegenerate we can find $n-k-1$ fibers $L_i$ of $W$ such that the curves $Y \cup L, L_1, \dots, L_{n-k-1}$ span a hyperplane $H \cong \PP^{n-1}$. By construction $H$ contains the curve
$$C=Y \cup L \cup L_1 \cup \dots \cup L_{n-k-1}$$
We get that $W$, an irreducible surface of degree $n-1$, contains a curve of at least degree $n$ as a hyperplane section and this is a contradiction.
\end{proof}

\begin{lemma} \label{L6}
Suppose $S=\Bl_p \PP^2$ and $\phi: S \to \PP^{g-1}$ is embedded by the complete linear series $\lvert E + (\frac{g-1}{2})F \rvert$. If $C$ is a smooth curve with
$[C]=E+dF$ where $1 \leq d \leq \frac{g-1}{2}$, then $\phi(C)$ is a rational normal curve of degree $k=\frac{g+2d-3}{2}$ sitting in some linear space $\Lambda \cong \PP^k \subset \PP^{g-1}$ and we have
$$N_{S/\PP^{g-1}} \rvert_C \cong N_{C/\PP^{k}} \oplus \O_{\PP^1}(k+1)^{\oplus g-k-2}$$
\end{lemma}

\begin{proof}
If $C \subset S$ is smooth with $[C]=E+dF$ then by adjunction we have
$$2g(C)-2=((d-3)F-E)(dF+E)=-2$$
so that $C$ is rational. Let $r=\frac{g-1}{2}$ and consider the exact sequence
\[\begin{tikzcd}
	0 & {\O_{S}((r-d)F)} & {\O_S(E+rF)} & {\O_C(E+rF)} & 0
	\arrow[from=1-1, to=1-2]
	\arrow[from=1-2, to=1-3]
	\arrow[from=1-3, to=1-4]
	\arrow[from=1-4, to=1-5]
\end{tikzcd}\]
which combined with the fact that $h^1((r-d)F)=0$ implies that the map 
$$H^0(\O_S(E+rF)) \to H^0(\O_C(E+rF))$$ 
is surjective. This allows us to compute the dimension of $H^0(\O_C(E+rF))$.
$$h^0(\O_C(E+rF))=h^0(\O_S(E+rF))-h^0(\O_S((r-d)F))=$$
$$(2r+1)-(r-d+1)=r+d=k+1$$
We also know that the degree of the linear series $\mathscr{D}$ on $C$ given by restricting $\lvert E+rF \rvert$ is
$$(E+rF)(E+dF)=r+d-1=k$$
Then $\mathscr{D}$ is a linear series on $C \cong \PP^1$ of degree and dimension $k$. It follows that $\mathscr{D}$ is the complete linear series associated to $\O_{\PP^1}(k)$. Therefore $\lvert E+rF \rvert$ maps $C$ to a rational normal curve in some linear subspace $\Lambda \cong \PP^{k} \subset \PP^{g-1}$. We have a commutative diagram of the form
\[\begin{tikzcd}
	0 & {T_C} & {T_S \rvert_C} & {N_{C/S}} & 0 \\
	0 & {T_{\PP^{k}} \rvert_C} & {T_{\PP^{g-1}} \rvert_C} & {\O_{\PP^1}(k)^{\oplus g-k-1}} & 0
	\arrow[from=1-1, to=1-2]
	\arrow[from=1-2, to=1-3]
	\arrow[from=1-2, to=2-2]
	\arrow[from=1-3, to=1-4]
	\arrow[from=1-3, to=2-3]
	\arrow[from=1-4, to=1-5]
	\arrow[from=1-4, to=2-4]
	\arrow[from=2-1, to=2-2]
	\arrow[from=2-2, to=2-3]
	\arrow[from=2-3, to=2-4]
	\arrow[from=2-4, to=2-5]
\end{tikzcd}\]
We claim the right hand map $N_{C/S} \to \O_{\PP^1}(k+1)^{\oplus g-k-1}$ is injective. This map is induced by the inclusion $T_{S} \rvert_C \to T_{\PP^{g-1}} \rvert_C$ and if there is a $p$ such that the map on fibers
$$N_{C/S,p} \to \C^{g-k-1}$$
is zero then we would have the inclusion of tangent spaces $T_{S,p} \subset T_{\Lambda,p}$. Recall that $S$ is a projective bundle over $\PP^1$ and let $\pi:S \to \PP^1$ be the projection map. If $\pi^{-1}(x)$ is the fiber containing $p$ then $T_{S,p} \subset T_{\Lambda,p}$ would imply that $\pi^{-1}(x) \subset \Lambda$ which contradicts Lemma \ref{L3}. The injectivity of $N_{C/S} \to \O_{\PP^1}(k+1)^{g-k-1}$ and the snake Lemma implies that there is an exact sequence
\[\begin{tikzcd}
	0 & {N_{C/\PP^k}} & {N_{S/\PP^{g-1}} \rvert_C} & {\mathcal{Q}} & 0
	\arrow[from=1-1, to=1-2]
	\arrow[from=1-2, to=1-3]
	\arrow[from=1-3, to=1-4]
	\arrow[from=1-4, to=1-5]
\end{tikzcd}\]
where $\mathcal{Q}$ is the cokernel of the map $N_{C/S} \to \O_{\PP^1}(k+1)^{\oplus g-k-1}$. To finish the proof it suffices to show that 
$$\mathcal{Q} \cong \O_{\PP^1}(k+1)^{\oplus g-k-2}$$
since this isomorphism and the calculation
$$\Ext^1(\mathcal{Q},N_{C/\PP^{k}}) = \Ext^1(\O_{\PP^1}(k+1)^{\oplus g-k-2},\O_{\PP^1}(k+2)^{\oplus k-1}) \cong
$$
$$\Ext^1({\O_{\PP^1}}^{\oplus g-k-2},\O_{\PP^1}(1)^{\oplus k-1}) \cong H^1(\O_{\PP^1}(1))^{\oplus (g-k-2)(k-1)}=0$$
implies the claimed splitting of $N_{S/\PP^{g-1}} \rvert_C$.
Assume that we have identified $S$ with the blowup of $\PP^2$ at the point $p=[0:0:1]$. Let $f$ be the equation of the curve $D \subset \PP^2$ whose strict transform is $C \subset \Bl_p \PP^2$. The degree $r$-forms
$$x^{r-d}f,x^{r-d-1}yf,\dots,y^{r-d}f$$
are linearly independent sections of $H^0(S,E+rF)$. Thus we can choose 
$$g_1, \dots, g_{k+1} \in H^0(S,E+rF)$$ 
such that the forms
$$x^{r-d}f,x^{r-d-1}yf, \dots, y^{r-d}f,g_1, \dots, g_{k+1}$$
give a basis for $H^0(S,E+rF)$. In other words the map
$$[x:y:z] \mapsto [x^{r-d}f: \dots: y^{r-d}f:g_1: \dots : g_{k+1}]$$
is an embedding of $S$ in $\PP^{g-1}$. With this choice of coordinates we have $\Lambda=V(z_0, \dots, z_{r-d})$ and
$$N_{\Lambda/\PP^{g-1}} \rvert_{C} = N_{\Lambda_0/\PP^{g-1}} \rvert_C \oplus \dots \oplus N_{\Lambda_{r-d}/\PP^{g-1}} \rvert_C$$
where $\Lambda_i=V(z_i)$ for $0 \leq i \leq r-d$. Furthermore for each $i$ we have a morphism
$$N_{C/S} \to N_{\Lambda_i/\PP^{g-1}} \rvert_C$$
which with respect to our coordinates is induced by the map $\O_{\PP^1} \to \O_{\PP^1}(k-2d+1)$ defined by $1 \mapsto x^{r-d-i}y^i$.
This shows that the map $N_{C/S} \to N_{\Lambda/\PP^{g-1}} \rvert_C = \O_{\PP^1}(k)^{\oplus g-k-1}$ is given by the $g-k-1=r-d+1$ forms $x^{r-d}, x^{r-d-1}y, \dots, y^{r-d}$, i.e. we have an exact sequence
\[\begin{tikzcd}
	0 & {\O_{\PP^1}(2d-1)} &&& {\O_{\PP^1}(k)^{\oplus g - k -1}} & {\mathcal{Q}} & 0
	\arrow[from=1-1, to=1-2]
	\arrow["{(x^{r-d},x^{r-d-1}y, \dots, y^{r-d})}", from=1-2, to=1-5]
	\arrow[from=1-5, to=1-6]
	\arrow[from=1-6, to=1-7]
\end{tikzcd}\]
dualizing we get an exact sequence
\[\begin{tikzcd}
	0 & {\mathcal{Q}^\vee} & {\O_{\PP^1}(-k)^{\oplus g - k -1}} & {\O_{\PP^1}(1-2d)} & 0
	\arrow[from=1-1, to=1-2]
	\arrow[from=1-2, to=1-3]
	\arrow[from=1-3, to=1-4]
	\arrow[from=1-4, to=1-5]
\end{tikzcd}\]
where the map $\Phi:\O_{\PP^1}(-k)^{\oplus g-k-1} \to \O_{\PP^1}(1-2d)$ is given by
$$(a_1,\dots,a_{g-k-1}) \mapsto \sum a_i x^{r-d-i}y^i$$
The morphism 
$$\O_{\PP^1}(-k-1)^{\oplus g - k - 2} \to \O_{\PP^1}(-k)^{\oplus g - k - 1}$$
given by the $(g-k-1) \times (g-k-2)$-matrix
$$
\begin{pmatrix}
y & 0 & 0 & \dots & 0 \\
-x & y & 0 & \dots & 0 \\
0 & -x & y & \dots & 0 \\
\vdots & \vdots & \vdots & \ddots & \vdots \\
0 & 0 & \dots & \ddots & y \\
0 & 0 & \dots & \dots & -x
\end{pmatrix}
$$
is injective and maps $\O_{\PP^1}(-k-1)^{\oplus g -k -2}$ onto $\ker(\Phi)=\mathcal{Q}^\vee$ so we conclude that $$\mathcal{Q} \cong \O_{\PP^1}(k+1)^{\oplus g - k - 1}$$ as desired.
\end{proof}

Using an analogous argument we get a similar result in the case when $S=\PP^1 \times \PP^1$.

\begin{lemma} \label{L7}
Suppose $S=\PP^1 \times \PP^1$ and $\phi:S \to \PP^{g-1}$ is embedded by the complete linear series $\lvert E + (\frac{g-2}{2})F \rvert$. If $C$ is a smooth curve with $[C]=E+dF$ where $1 \leq d \leq \frac{g-2}{2}$. Then $\phi(C)$ is a rational normal curve of degree $k=\frac{g+2d-2}{2}$ sitting in some linear space $\Lambda \cong \PP^{k} \subset \PP^{g-1}$ and we have
$$N_{S/\PP^{g-1}} \rvert_C \cong N_{C/\PP^{k}} \oplus \O_{\PP^1}(k+1)^{\oplus g-k-2}$$
\end{lemma}

%


\section{The Destabilizing Subbundle}\label{sec:DSB}

The goal of this section is to show that if $C$ is a trigonal canonical curve then $N_{C/\PP^{g-1}}$ is not semi-stable. Recall that if $S \subset \PP^{g-1}$ is smooth of dimension $k \geq 2$ and $C \subset S$ then $N_{C/S}$ is a rank $k-1$ subbundle of $N_{C/\PP^{g-1}}$. Therefore to produce a destabilizing line subbundle of $N_{C/\PP^{g-1}}$ we exhibit a smooth surface $S$ containing $C$.

\subsection{The Surface Scroll}

Given any canonical curve $C$, Petri's theorem tells us that the homogeneous ideal of $C \subset \PP^{g-1}$ is generated by quadrics unless $C$ is trigonal or a smooth plane quintic. Furthermore, even when $C$ is trigonal there are always many independent quadrics vanishing on $C$ due to Max Noether's Theorem which states that if $C$ is a non-hyperelliptic curve and $K$ is a canonical divisor of $C$ then the homomorphisms
$$\Sym^l H^0(C, K) \to H^0(C, K^l)$$
are surjective for $l \geq 1$. A straightforward computation using the case $l=2$ of this result shows that a canonical curve $C \subset \PP^{g-1}$ is contained in $(g-2)(g-3)/2$ linearly independent quadrics. For example, a genus $5$ canonical curve $C$ lies on $3$ independent quadrics and the general such $C$ is a complete intersection of these quadrics. However, if $C$ is trigonal the quadrics intersect in a cubic scroll containing $C$. The following Proposition from \cite{ACGH} shows that a similar phenomenon occurs for trigonal curves of higher genus.

\begin{proposition} \label{P7}
If the intersection of the quadrics containing a canonical surface $C$ contains a point $p \notin C$, then $C$ lies on either the Veronese surface (in case g=6) or on a (smooth) rational normal scroll.
\end{proposition}

By the Proposition a trigonal canonical curve of genus $g \geq 5$ lies on a rational normal scroll $S$ of dimension $2$. Since rational normal scrolls are minimal degree varieties and $S \subset \PP^{g-1}$ we must have $\deg(S)=g-2$. Furthermore using geometric Riemann-Roch (page 12 on \cite{ACGH}) we see that the fibers $\psi^{-1}(t)$ of the degree $3$ map $\psi: C \to \PP^1$ all lie on lines in $\PP^{g-1}$. As $t$ varies in $\PP^1$ these lines sweep out the surface $S$, in particular $S$ is a ruled surface over $\PP^1$.

\subsection{The Destabilizing Subbundle}

To end the section we will recall a result from \cite{HL20} which computes the class of $C$ in $S$. Before stating the result we need to briefly discuss the moduli space of trigonal curves. The locus of smooth trigonal curves will be denoted by $\mathfrak{T}_g$ and $\overline{\mathfrak{T}}_g$ will denote its closure in $\overline{\mathcal{M}}_g$. Recall that every ruled surface over a curve is the projectivization of a vector bundle. Since the surface scroll $S$ containing a trigonal canonical curve is a ruled surface of degree $g-2$ in $\PP^{g-1}$ we have
$$S \cong \PP(\O_{\PP^1}(a) \oplus \O_{\PP^1}(b))$$
where $b \geq a \geq 0$ and $a+b=g-2$. The difference $n=b-a$ is called the \textbf{Maroni invariant} of the trigonal curve $C$. By definition a trigonal curve of Maroni invariant $n$ lies on the Hirzebruch surface $\mathbb{F}_n$. 

We can describe the vector bundle $V=\O_{\PP_1}(a)\oplus \O_{\PP^1}(b)$ whose projectivization is $S$. If $C$ is a trigonal curve and $f:C \to \PP^1$ is a degree $3$ map then $f_* \O_C$ is a rank $3$ vector bundle on $\PP^1$. Furthermore there is an injection $\O_{\PP^1} \hookrightarrow f_* \O_C$ and the cokernel of this inclusion is $V$. The following result from section $12$ of \cite{ZSF00} tells us that the general trigonal curve has $n=0$ or $n=1$.

\begin{proposition} \label{P9}
For a general trigonal canonical curve $C$ the vector bundle $V$ is balanced, i.e. the integers $a$ and $b$ are equal or $1$ apart according to $g \mod 2$.
\end{proposition}

Due to Proposition \ref{P9} when proving Theorem \ref{MT} we may assume that $n=0$ if $g$ is even or $n=1$ if $g$ is odd. The following result from section $3$ of \cite{HL20} computes the class of $C \subset S \cong \mathbb{F}_n$ in terms of the genus and Maroni invariant.

\begin{proposition} \label{P10}
If $C \subset \mathbb{F}_n$ is a trigonal curve with Maroni invariant $n$ then
$$[C]=3E+\left(\frac{g+3n}{2}+1\right)F$$
\end{proposition}

Note that since $S \cong \mathbb{F}_n$ we have $E^2=-n=a-b$, $F^2=0$ and $E\cdot F=1$. Now using Proposition \ref{P10} we can compute the degree of $N_{C/S}$.
$$\deg(N_{C/S})=[C]^2=3g+6$$
On the other hand from Proposition \ref{P1} we know the slope of $N_{C/\PP^{g-1}}$.
$$\mu(N_{C/\PP^{g-1}})=2g+4+\frac{6}{g-2}$$
Furthermore when $g \geq 4$ we have
$$(g+2)(g-2) \geq 6$$
which implies that $N_{C/S}$ is a destabilizing line subbundle of $N_{C/\PP^{g-1}}$.

\begin{remark}
The result of Theorem \ref{MT} does not apply if $g=3$ and is already well known in the case $g=4$.
\begin{enumerate}
\item Any canonical curve $X \subset \PP^2$ of genus $3$ is a smooth plane quartic, in particular any such $X$ has gonality $3$. The normal bundle of $X \subset \PP^2$ is the line bundle $\mathscr{O}_C(4)$ which is stable.
\item The normal bundle $N_{C/\PP^3}$ of any canonical curve $C \subset \PP^3$ of genus $4$ is easily computed due to the fact that such a $C$ is a complete intersection of a quadric $Q$ and a cubic $Y$.
$$N_{C/\PP^3} \cong N_{C/Q} \oplus N_{C/Y} = \O_C(2) \oplus \O_C(3)$$
The inclusion $N_{C/S} \subset N_{C/\PP^{3}}$ is given by
$$\O_C(3) \subset \O_C(2) \oplus \O_C(3)$$
which is the Harder-Narasimhan filtration for $N_{C/\PP^{3}}$ since $\deg(\O_C(3))>\deg(\O_C(2))$.
\end{enumerate}
\end{remark}

By the remark we see that it suffices to prove Theorem \ref{MT} in the case when $g \geq 5$. Also from Proposition \ref{P1} and the short exact sequence
\[\begin{tikzcd}
	0 & {N_{C/S}} & {N_{C/\PP^{g-1}}} & {N_{S/\PP^{g-1}} \rvert_C} & 0
	\arrow[from=1-1, to=1-2]
	\arrow[from=1-2, to=1-3]
	\arrow[from=1-3, to=1-4]
	\arrow[from=1-4, to=1-5]
\end{tikzcd}\]
we learn that $N_{S/\PP^{g-1}} \rvert_C$ has degree $2g^2-3g-8$ and rank $g-3$ so that
$$\mu(N_{S/\PP^{g-1}} \rvert_C)=\frac{2g^2-3g-8}{g-3}=2g+3+\frac{1}{g-3}$$
In particular
$$2g^2-3g-8-(2g+3)(g-3)=1$$
which implies $gcd(rk(N_{S/\PP^{g-1}} \rvert_C),deg(N_{S/\PP^{g-1}} \rvert_C))=1$ hence $N_{S/\PP^{g-1}} \rvert_C$ is semi-stable iff it is stable.

\section{Proof of Main Theorem}

The result of Theorem \ref{MT} is equivalent to showing that $N_{S/\PP^{g-1}} \rvert_C$ is semi-stable for a general trigonal canonical curve $C \subset \PP^{g-1}$. The idea is to degenerate $C$ to a union of rational curves $X=C_1 \cup C_2 \cup C_3$ and show that $N_{S/\PP^{g-1}} \rvert_X$ is semi-stable with respect to the adjusted slope. Then we will use Proposition \ref{P5} to conclude that $N_{S/\PP^{g-1}} \rvert_{C}$ is semi-stable for a general $C$.

\subsection{Setup and Notation}
Let $\mathcal{H}$ be the component of $\Hilb_{(2g-2)t+1-g}(\PP^{g-1})$ containing smooth curves in $\PP^{g-1}$ of genus $g$ and degree $2g-2$. Denote by $\mathcal{T}$ the closure of the locus of smooth curves in $\mathcal{H}$ which admit a degree $3$ map to $\PP^1$. We know from Proposition \ref{P9} that there is an open subset of $\mathcal{T}$ on which the Maroni invariant is $n=0$ or $n=1$ depending on $g \mod 2$. Therefore in our proof we may assume that $S \cong \mathbb{F}_0 \cong \PP^1 \times \PP^1$ when $g$ is even and $S \cong \mathbb{F}_1 \cong \Bl_p \PP^2$ when $g$ is odd. The embedding of $S$ in $\PP^{g-1}$ is given by the complete linear series $\lvert E + \lceil\frac{g-2}{2}\rceil F\rvert$ which restricts to the canonical linear series on $C$. By Proposition \ref{P10} the class of $C$ in $S$ is
$$[C]=3E+\left(\frac{g+2}{2}\right)F$$
when $g$ is even and the class is
$$[C]=3E+\left(\frac{g+5}{2}\right)F$$
when $g$ is odd.

Furthermore, given any such $C$ we can find a flat family $T$ over $\PP^1$ that degenerates $C$ to a union $X=C_1 \cup C_2 \cup C_3$ where
$$[C_1]=E+\biggl\lceil\frac{g-4}{2}\biggr\rceil F$$
$$[C_2]=E+2F$$
and $[C_3]=E+2F$ if $g$ is odd or $[C_3]=E+F$ if $g$ is even. If we can show that the bundle $N_{S/\PP^{g-1}} \rvert_X$ is semi-stable with respect to the adjusted slope, then by Proposition \ref{P5} the bundle $N_{S/\PP^{g-1}} \rvert_{X_t}$ is semi-stable for the general member $X_t$ of $T$. This implies the result of Theorem \ref{MT} since each smooth curve $C$ in $\mathcal{T}$ can be deformed to such an $X$.

Let $\mathcal{E}=N_{S/\PP^{g-1}} \rvert_X$ and $\nu:\tilde{X} \to X$ the normalization of $X$. To calculate the adjusted slope of a possible subbundle $\mathcal{F} \subset \nu^* \mathcal{E}$ we first need to compute the restriction $\nu^* \mathcal{E} \rvert_{\tilde{C_i}}$ to the components $\tilde{C_i}$ of the normalization $\tilde{X}$. Observe that $\nu^* \mathcal{E} \rvert_{\tilde{C_i}}= N_{S/\PP^{g-1}} \rvert_{C_i}$ and by Lemmas \ref{L6} and \ref{L7} we have
$$N_{S/\PP^{g-1}} \rvert_{C_1} = N_{C_1/\Lambda_1} \oplus \O_{\PP^1}(g-2)$$
$$N_{S/\PP^{g-1}} \rvert_{C_2} = N_{C_2/\Lambda_2} \oplus \O_{\PP^1}\left(\biggl\lfloor\frac{g+4}{2}\biggr\rfloor \right)^{\oplus \lceil\frac{g-6}{2}\rceil}$$ 
$$N_{S/\PP^{g-1}} \rvert_{C_3} = N_{C_3/\Lambda_3} \oplus \O_{\PP^1}\left(\biggl\lceil\frac{g+2}{2}\biggr\rceil \right)^{\oplus \lfloor\frac{g-4}{2}\rfloor}$$
where the $\Lambda_i$ are linear spaces such that
$$\lambda_1:=\dim(\Lambda_1)=g-3$$
$$\lambda_2:=\dim(\Lambda_2)=\biggl\lfloor \frac{g+2}{2}\biggr\rfloor$$
$$\lambda_3:=\dim(\Lambda_3)=\biggl\lceil \frac{g}{2}\biggr\rceil$$
and for each $i$ we have $C_i$ is a rational normal curve of degree $\lambda_i$ lying in $\Lambda_i$.
Given a rank $r$ subbundle $\mathcal{F} \subset \nu^* \mathcal{E}$, where $1 \leq r \leq g-4$, we have
$$\mathcal{F} \rvert_{\tilde{C_i}} \subset N_{S/\PP^{g-1}} \rvert_{C_i} = N_{C_i/\Lambda_i} \oplus \O_{\PP^1}(\lambda_i+1)^{\oplus g-\lambda_i-2}$$
which implies that $\mathcal{F} \rvert_{\tilde{C_i}}$ splits as a direct sum
$$\mathcal{F} \rvert_{\tilde{C_i}}=\mathcal{H}_i \oplus \mathcal{M}_i$$
where $\mu(\mathcal{M}_i) \leq \lambda_i+1$ and for some integers $a_i$ we have
$$\mathcal{H}_i \cong \O_{\PP^1}(\lambda_i+2)^{\oplus a_i} \subset N_{C_i/\Lambda_i}$$ 
Note that $\mu(N_{S/\PP^{g-1}} \rvert_{X})=2g+3+\frac{1}{g-3}$ so Theorem $\ref{MT}$ will be proven if we can rule out a subbundle $\mathcal{F}$ with $\mu^{\adj}(\mathcal{F}) > 2g+3$, this can be done in four steps.
\begin{enumerate}
\item Provide a bound on $\sum a_i$ which ensures that $\mathcal{F}$ does not destabilize $\nu^*\mathcal{E}$.
\item Show that there is no destabilizing subbundle $\mathcal{F} \subset \nu^*\mathcal{E}$ such that $a_1=r$. This combined with the aforementioned bound on $\sum a_i$ will show that $\nu^* \mathcal{E}$ has no destabilizing line subbundle.
\item Next rule out a rank $2$ destabilizing line subbundle $\mathcal{F}$.
\item Finally rule out a rank $r \geq 3$ subbundle. 
\end{enumerate}

\subsection{Preliminary Degree Bound}

In view of the splittings
$$\mathcal{F} \rvert_{\tilde{C_i}} = \mathcal{H}_i \oplus \mathcal{M}_i \cong \O_{\PP^1}(\lambda_i+2)^{\oplus a_i} \oplus \mathcal{M}_i$$
we get a bound on $\deg(\mathcal{F})$ in terms of the sum of the $a_i$.
$$\deg(\mathcal{F})=\sum_{i=1}^3 \deg(\mathcal{F} \rvert_{\tilde{C_i}})=$$
$$a_1(g-1)+a_2\biggl\lfloor\frac{g+6}{2}\biggr\rfloor+a_3\biggl\lceil \frac{g+4}{2} \biggr\rceil+\sum_{i=1}^3 \deg(\mathcal{M}_i) \leq$$
$$a_1(g-1)+a_2\biggl\lfloor\frac{g+6}{2}\biggr\rfloor+a_3\biggl\lceil \frac{g+4}{2} \biggr\rceil+(r-a_1)(g-2)+(r-a_2)\biggl\lfloor\frac{g+4}{2}\biggr\rfloor+(r-a_3)\biggl\lceil\frac{g+2}{2}\biggr\rceil=$$
$$r(2g+1)+\sum_{i=1}^3 a_i$$
and dividing by $r$ yields
$$\mu(\mathcal{F}) \leq 2g+1+\frac{\sum a_i}{r}$$
If $\sum a_i \leq 2r$ then
$$\mu^{\adj}(\mathcal{F}) \leq \mu(\mathcal{F}) \leq 2g+3$$
as desired. Thus we are reduced to the case when $3r \geq \sum a_i \geq 2r+1$. Recall that
$$\mu^{\adj}(\mathcal{F})=\mu(\mathcal{F})-\frac{\delta_{\mathcal{F}}}{r}
$$
where we let
$$\delta_{\mathcal{F}}=\sum_{x \in \Sing(X)} \codim_{\mathcal{F}}(\mathcal{F}_{x_1} \cap \mathcal{F}_{x_2})$$
and $x_1,x_2$ are the points lying above $x$ in the normalization. The goal is to give suitably large lower bounds for $\delta_{\mathcal{F}}$ in the cases when $\sum a_i \geq 2r+1$. 

\subsection{Ruling out $\bf{a_1=r}$}

Our next goal is to rule out a subbundle $\mathcal{F} \subset \nu^* \mathcal{E}$ such that the following equality holds. 
$$\mathcal{F} \rvert_{\tilde{C_1}} = \mathcal{H}_1 \cong \O_{\PP^1}(\lambda_1+2)^{\oplus r}$$ 
Before doing this we need to introduce some notation and terminology. Given a rational normal curve $R \subset \PP^r$ of degree $r$ fix a vector space $V_R$ of dimension $r-1$. Then a degree $r+2$ rank $k$ subbundle $\mathcal{Q} \subset N_{R/\PP^r}$ is equivalent to giving a map of vector bundles $\O_{R}^{\oplus k} \to \O_{R} \otimes V_R$, i.e. equivalent to specifying a $k$ dimensional subspace $W_{\mathcal{Q}} \subset V_R$. We will refer to $W_{\mathcal{Q}}$ as the \textbf{subspace in} $V_R$ \textbf{corresponding to} $\mathcal{L}$. 

The curve $X=C_1 \cup C_2 \cup C_3$ has three nodal singularities $p_1,p_2,p_3$ at the $3$ intersection points of $C_2,C_3$. Let $p_{i,2}$ and $p_{i,3}$ be the points lying above $p_i$ in the normalization $\tilde{X}$. Note that the points $p_1, p_2, p_3$ span a $\PP^2 \subset \PP^{g-1}$ which is the intersection of the linear spaces $\Lambda_2$ and $\Lambda_3$. For each $j=1,2,3$ the image of the natural map on fibers over $p_j$
$$T_{p_j} \PP^2 \to N_{S/\PP^{g-1},p_j}$$
is exactly the two dimensional intersection of the fibers $N_{C_2/\Lambda_2,p_j}$ and $N_{C_3/\Lambda_3,p_j}$ in $N_{S/\PP^{g-1},p_j}$, we will use $T_j$ to denote this two dimensional intersection.
Set $\kappa_2=\lfloor\frac{g}{2} \rfloor$ and denote by $y_1, \dots, y_{\kappa_2}$ the $\kappa_2$ intersection points of $C_1$ and $C_2$. For each $i$ we have points $y_{i,1} \in \tilde{C_1}$ and $y_{i,2} \in \tilde{C_2}$ lying above $y_i$ in the normalization $\tilde{X}$ of $X$. Similarly we set $\kappa_3=\lceil \frac{g-2}{2}\rceil$ and denote by $z_1, \dots, z_{\kappa_3}$ the intersection points of $C_1$ and $C_3$ and $z_{i,1}, z_{i,3}$ the points in $\tilde{X}$ lying above $z_i$. The points $y_i$ span a linear space $\Gamma$ which is exactly the intersection of $\Lambda_1$ and $\Lambda_2$. For each $i$ the image of the composition of maps
$$T_{\Gamma,y_i} \to T_{\PP^{g-1},y_i} \to N_{C_2/\Lambda_2,y_i}$$
equals the fiber over $y_i$ of the the direct sum of pointing bundles
$$N_{C_2 \to y_1} \oplus \dots \oplus N_{C_2 \to y_{i-1}} \oplus N_{C_2 \to y_{i+1}} \oplus \dots \oplus N_{C_2 \to y_{\kappa_2}}$$
We will denote this fiber by $\Gamma_i$, note that $\Gamma_i$ is the intersection of the fibers  $N_{C_1/\Lambda_1,y_i}$ and $N_{C_2/\Lambda_2,y_i}$ in $N_{S/\PP^{g-1},y_i}$. Lastly, choose vectors $v_{y_1}, \dots v_{y_{\kappa_2}} \in V_{C_2}$ such that $\langle v_{y_i} \rangle$ corresponds to the pointing bundle $N_{C_2 \to y_i}$.

Assume that $a_1=r$, i.e. we have
$$\mathcal{F} \rvert_{\tilde{C_1}} \subset N_{C_1/\Lambda_1}$$
which with our notation is equivalent to
$$\mathcal{F} \rvert_{\tilde{C_1}} = \mathcal{H}_1$$
We claim that the following inequalities hold
\begin{equation} \label{(1)} \sum_{i=1}^{\kappa_2} \codim_{\mathcal{F}}(\mathcal{F}_{y_{i,2}} \cap \mathcal{F}_{y_{i,1}}) \geq a_2
\end{equation}
\begin{equation} \label{(2)}
\sum_{i=1}^{\kappa_3} \codim_{\mathcal{F}}(\mathcal{F}_{z_{i,3}} \cap \mathcal{F}_{z_{i,1}}) \geq a_3
\end{equation}
so $\delta_{\mathcal{F}} \geq a_2+a_3$. We will only show the first inequality (\ref{(1)}) since the same strategy with slightly differing numerics works to prove both. To start note that
$$\dim(\mathcal{H}_{2,y_i} \cap \mathcal{F}_{y_{i,1}}) \leq \dim(\Gamma_i) = \biggl\lfloor\frac{g-2}{2}\biggr\rfloor$$
so that if $\mathcal{H}_2=N_{C_2/\Lambda_2}$ then $\dim(\mathcal{H}_{2,y_i}) > \lfloor \frac{g-2}{2} \rfloor$ which implies that $\mathcal{H}_{2,y_i}$ is not a subspace of $\mathcal{F}_{y_{i,1}}$. In other words
$$\dim(\mathcal{F}_{y_{i,2}} \cap \mathcal{F}_{y_{i,1}}) < r$$
for all $i$ and we conclude
$$\sum_{i=1}^{\kappa_2} \codim_{\mathcal{F}}(\mathcal{F}_{y_{i,2}} \cap \mathcal{F}_{y_{i,1}}) \geq \kappa_2 = a_2$$
We are reduced to the case when $\rk(\mathcal{H}_2)=a_2 \leq \lfloor \frac{g-2}{2} \rfloor$. Given any point $y_i$ with
$$\dim(\mathcal{F}_{y_{i,2}} \cap \mathcal{F}_{y_{i,1}}) = r$$ we would have $W_{\mathcal{H}_2} \subset U_i$ where $W_{\mathcal{H}_2}$ corresponds to $\mathcal{H}_2$ and $U_i$ is the subspace of $V_{C_2}$ spanned by
$$v_{y_1} \cdots v_{y_{i-1}}, v_{y_{i+1}}, \cdots, v_{y_{\kappa_2}}$$
Given any integers
$$1 \leq i_1 < i_2 < \cdots < i_k \leq \kappa_2$$
we have
$$\dim(U_{i_1} \cap \dots \cap U_{i_k})=\biggl\lfloor\frac{g-2}{2} \biggr\rfloor-k$$
by Lemma \ref{L5}. Hence it follows that $W_{\mathcal{H}_2} \subset U_i$ for at most $\lfloor \frac{g-2}{2} \rfloor-a_2$ of the points $y_i$. Thus there are at least $a_2+1$ points $y_j$ such that
$$\dim(\mathcal{F}_{y_{j,1}} \cap \mathcal{F}_{y_{j,2}}) < r$$
which implies that 
$$\sum_{i=1}^{\kappa} \codim_{\mathcal{F}}(\mathcal{F}_{y_{i,2}} \cap \mathcal{F}_{y_{i,1}}) \geq a_2+1$$ 
Putting this all together we conclude that inequality (\ref{(1)}) holds.

By summing the inequalities (\ref{(1)}) and (\ref{(2)}) we get $\delta_{\mathcal{F}} \geq a_2 + a_3$ which combined with the assumption $a_1=r$ and our previous bound on the adjusted slope gives
$$\mu^{\adj}(\mathcal{F}) \leq 2g+2 + \frac{a_2+a_3-\delta_{\mathcal{F}}}{r} \leq 2g+2 < 2g+3$$
as desired. From now on we can assume that $a_1 < r$ which on its own implies that $\nu^* \mathcal{E}$ has no destabilizing line subbundle.

\subsection{Rank 2}
Next we will rule out a destabilizing subbundle $\mathcal{F} \subset \nu^* \mathcal{E}$ of rank $r=2$. When $r=2$ we get the inequality $6 \geq \sum a_i \geq 5$. Since we can assume $a_1 < 2$ we must have $\sum a_i = 5$ and $a_2=a_3=2$. In other words we have
$$\mathcal{F} \rvert_{\tilde{C_2}} = \mathcal{H}_2 \subset N_{C_2/\Lambda_2}$$
$$\mathcal{F} \rvert_{\tilde{C_3}} = \mathcal{H}_3 \subset N_{C_3/\Lambda_3}$$
Suppose $\langle w_1,w_2 \rangle$ is the subspace of $V_{C_2}$ associated to $\mathcal{F} \rvert_{\tilde{C_2}}$. For each $i$ we let $\langle v_{p_i} \rangle$ be the subspace of $V_{C_2}$ associated to the pointing bundle $N_{C_2 \to p_i}$. The $v_{p_i}$ are linearly independent by Lemma \ref{L5} so that 
$$\langle v_{p_1}, v_{p_2} \rangle \cap \langle v_{p_1}, v_{p_3} \rangle \cap \langle v_{p_2}, v_{p_3} \rangle = (0)$$
hence for $j=1,2$ at least one of these subspaces does not contain $w_j$.

Suppose this subspace is the same for both $j=1,2$, i.e. WLOG we have
$$\langle w_1, w_2 \rangle \cap \langle v_{p_2}, v_{p_3} \rangle = (0)$$
Then $\mathcal{F}_{p_{1,2}} \cap T_1 = (0)$ which implies $\mathcal{F}_{p_{1,2}} \cap \mathcal{F}_{p_{1,3}}=(0)$, i.e. $\delta_{\mathcal{F}} \geq 2$ and
$$\mu^{\adj}(\mathcal{F}) \leq 4g+6$$ 
as desired. If we instead have WLOG that $w_1 \notin \langle v_{p_2}, v_{p_3} \rangle$ and $w_2 \notin \langle v_{p_1}, v_{p_3} \rangle$ then we get
$$\dim(\mathcal{F}_{p_{1,2}} \cap T_1) \leq 1$$
$$\dim(\mathcal{F}_{p_{2,2}} \cap T_2) \leq 1$$
which together give $\delta_{\mathcal{F}} \geq 2$ and
$$\mu^{\adj}(\mathcal{F}) \leq 4g+6$$

\subsection{Rank greater than or equal to 3}
It remains to rule out a destabilizing subbundle $\mathcal{F} \subset \nu^* \mathcal{E}$ with rank $r \geq 3$ and $a_1 < r$.
Recall that for each component $\tilde{C_i}$ of the normalization $\tilde{X}$ we have
$$\mathcal{F} \rvert_{\tilde{C_i}} \cong \mathcal{H}_i \oplus \mathcal{M}_i$$
where $\mu(\mathcal{M}_i) \leq \lambda_i+1$.
Let $b_i=\rk(\mathcal{M}_i)$ and assume that $b_2+b_3 \leq r-3$ so that $2+b_2+b_3 < r$. For the points of intersection $p_1,p_2,p_3$ of $C_2$ and $C_3$ we have
$$\dim(\mathcal{F}_{p_{i,2}} \cap \mathcal{F}_{p_{i,3}}) \leq 2+b_2+b_3$$
which implies that
$$\delta_{\mathcal{F}} \geq \sum_{i=1}^3 \codim_{\mathcal{F}}(\mathcal{F}_{p_{i,2}} \cap \mathcal{F}_{p_{i,3}}) \geq 3(r-b_2-b_3-2)=3r-3b_2-3b_3-6$$
thus we compute
$$r (\mu^{\adj}(\mathcal{F})) \leq r(2g+1)+\left(\sum a_i\right) - \delta_{\mathcal{F}} \leq$$
$$r(2g+1)+\left(\sum a_i\right) -3r+3b_2+3b_3 + 6 =$$
$$r(2g+1)+ a_1 + 2(b_2+b_3)+6 - r \leq r(2g+3) - 1 < r(2g+3)$$
as desired. We are reduced to the case when $b_2+b_3 \geq r-2$ which implies that $a_2+a_3 \leq r+2$. But then since $\sum a_i \geq 2r+1$ we must have $a_1 \geq r-1$ i.e. $a_1=r-1$ since we ruled out $a_1=r$ above. If $a_1=r-1$ then $a_2+a_3=r+2$ and $\sum a_i=2r+1$. Notice that $\sum a_i=2r+1$ implies
$$\mu^{\adj}(\mathcal{F}) \leq 2g+3+\frac{1}{r}$$
so we are reduced to the case when $r \mu(\mathcal{F})=r(2g+3)+1$. To finish the proof we need $\delta_{\mathcal{F}} \geq 1$ and for this it suffices to show there exists a singular point $x \in X$ with
$$\mathcal{F}_{x_1} \neq \mathcal{F}_{x_2}$$
where $x_1$ and $x_2$ are the points in the normalization $\tilde{X}$ lying above $x$. In particular we can assume that $\mathcal{F}_{x_1}=\mathcal{F}_{x_2}$ for all $x \in \Sing(X)$ and argue towards a contradiction. To rule out the case when $a_2=r$ (and by a similar argument rule out $a_3=r$) assume that $\mathcal{F} \rvert_{\tilde{C_2}} \cong \mathcal{H}_2 \subset N_{C_2/\Lambda_2}$.
Since $a_2+a_3=r+2$ we have $a_3=\rk(\mathcal{H}_3)=2$ and the assumption
$$\mathcal{F}_{p_{i,2}}=\mathcal{F}_{p_{i,3}}$$
implies that
$$\mathcal{H}_{3,p_i} \subset N_{C_2/\Lambda_2,p_i} \cap N_{C_3/\Lambda_3,p_i}$$
for all $i$. But this implies that for any distinct $i,j \in \{1,2,3\}$
$$W_{\mathcal{H}_3} = \langle v_{p_i},v_{p_j} \rangle$$
this is a contradiction because in particular
$$\langle v_{p_1},v_{p_2} \rangle \neq \langle v_{p_1},v_{p_3} \rangle$$
Therefore we may add $a_2 < r$ and $a_3 < r$ to our list of assumptions. This combined with our other reductions, in particular $a_2+a_3=r+2$, already rules at the case $r=3$. Thus we can also assume $r \geq 4$ and since $a_2 < r$, $a_3 < r$, $a_2+a_3=r+2$ it follows that $a_2 \geq 3$ and $a_3 \geq 3$. Our assumption that
$$\mathcal{F}_{p_{i,2}}=\mathcal{F}_{p_{i,3}}$$
for all $i$ implies that
$$T_i = N_{C_1/\Lambda_1,p_i} \cap N_{C_2/\Lambda_2,p_i} \subset \mathcal{H}_{2,p_i}$$
for all $i$. This is because otherwise we would have
$$\dim(\mathcal{H}_{2,p_i} \cap \mathcal{H}_{3,p_i}) < 2$$
so that
$$\dim(\mathcal{F}_{p_{i,2}} \cap \mathcal{F}_{p_{i,3}}) < 2+b_2+b_3=r$$
It follows that if $v_{p_1},v_{p_2},v_{p_3}$ are vectors in $V_{C_2}$ spanning the subspaces corresponding to $N_{C_2 \to p_1},N_{C_2 \to p_2},N_{C_2 \to p_3}$ and $W_{\mathcal{H}_2} \subset V_{C_2}$ is the subspace corresponding to $\mathcal{H}_2$ then 
$$\langle v_{p_i},v_{p_j} \rangle \subset W_{\mathcal{H}_2}$$
for each $i,j \in \{1,2,3\}$. In particular we see that $W_{\mathcal{H}_2}$ contains the $3$-dimensional subspace $\langle v_{p_1},v_{p_2},v_{p_3} \rangle$ which implies that
$$\mathcal{H}_2 \cong N_{C_2 \to p_1} \oplus N_{C_2 \to p_2} \oplus N_{C_2 \to p_3} \oplus \O_{\PP^1}\left(\biggl\lfloor \frac{g+6}{2} \biggr\rfloor\right)^{\oplus a_2-3}$$
An analogous argument shows that we may also assume
$$\mathcal{H}_3 \cong N_{C_3 \to p_1} \oplus N_{C_3 \to p_2} \oplus N_{C_3 \to p_3} \oplus \O_{\PP^1}\left(\biggl\lceil\frac{g+4}{2}\biggr\rceil\right)^{\oplus a_3-3}$$

For each $i$ let $w_{y_i}$ be a vector in $V_{C_1}$ which spans the subspace corresponding the pointing bundle $N_{C_1 \to y_i}$. We claim that the vector
$$l_i:= w_{y_1} + \dots + w_{y_{i-1}} + \hat{w_{y_i}} + w_{y_{i+1}} + \dots + w_{y_{\kappa_2}}$$ 
is contained in the subspace $W_{\mathcal{H}_1}$ corresponding to $\mathcal{H}_1$.
Due to the assumption
$$\mathcal{F}_{y_{i,1}}=\mathcal{F}_{y_{i,2}}$$
we must have
$$\dim(\mathcal{H}_{1,y_i} \cap (N_{C_2 \to p_1,y_i} \oplus N_{C_2 \to p_2,y_i} \oplus N_{C_2 \to p_3,y_i})) \geq 2$$
since otherwise
$$\dim(\mathcal{F}_{y_{i,1}} \cap \mathcal{F}_{y_{i,2}}) \leq \dim(\mathcal{H}_{1,y_i} \cap (N_{C_2 \to p_1,y_i} \oplus N_{C_2 \to p_2,y_i} \oplus N_{C_2 \to p_3,y_i}))+(r-2) < r$$
On the other hand
$$\mathcal{H}_{1,y_i} \cap (N_{C_2 \to p_1,y_i} \oplus N_{C_2 \to p_2,y_i} \oplus N_{C_2 \to p_3,y_i}) \subseteq N_{C_1/\Lambda_1,y_i} \cap (N_{C_2 \to p_1,y_i} \oplus N_{C_2 \to p_2,y_i} \oplus N_{C_2 \to p_3,y_i})$$
and this latter subspace is $2$ dimensional, so we conclude that the above inclusion of subspaces is actually an equality. The linear space spanned by the points $\{p_1,p_2,p_3\}$ intersects the linear space spanned by $\{y_1, \dots, y_{i-1},y_{i+1},y_{\kappa_2}\}$ in a point $q_i$, denote by $L_i$ the line $\overline{y_i,q_i}$. Observe that if $\Delta_i$ is the image of $T_{y_i} L_i$ in $N_{C_2/\Lambda_2,y_i}$ then
$$\Delta_i \subset N_{C_1/\Lambda_1,y_i} \cap (N_{C_2 \to p_1,y_i} \oplus N_{C_2 \to p_2,y_i} \oplus N_{C_2 \to p_3,y_i})$$
so that from the above we must have $\Delta_i \subset \mathcal{H}_{1,y_i}$. However we also have $\Delta_i=\mathcal{L}_{i,y_i}$ where $\mathcal{L}_i$ is the degree $g-1$ line subbundle of $N_{C_1/\Lambda_1}$ such that $l_i$ spans the subspace corresponding to $\mathcal{L}_i$, this implies that $l_i \in W_{\mathcal{H}_1}$ as claimed. 
Since the matrix with zeros on the diagonal and ones everywhere else is invertible it follows that the $l_i$ are independent and thus are a basis for
$$\langle w_{y_1}, \dots, w_{y_{\kappa_2}} \rangle$$
We conclude that
$$N_{C_1 \to y_1} \oplus \dots \oplus N_{C_1 \to y_{\kappa_2}} \subset \mathcal{H}_1$$
An analogous argument with $C_3$ in place of $C_2$ and the $z_i$ in place of the $y_i$ gives
$$N_{C_1 \to z_1} \oplus \dots \oplus N_{C_1 \to z_{\kappa_3}} \subset \mathcal{H}_1$$
But by Lemma \ref{L5} the bundles $N_{C_1 \to y_i}$ and $N_{C_1 \to z_j}$ span all of $N_{C_1/\Lambda_1}$ because $y_1,\dots, y_{\kappa},z_1,\dots,z_{\kappa}$ are $g-1$ distinct points of $C_1$. This is the desired contradiction because $\mathcal{H}_1$ has rank $r-1 < g-3$.

To summarize we have previously shown that the only possible destabilizing subbundles $\mathcal{F} \subset \nu^* \mathcal{E}$ have rank $r \geq 4$ and $\mu(\mathcal{F})=2g+3+1/r$. The above contradiction shows that given such a subbundle $\mathcal{F}$ there must exist a singular point $x \in X$ with
$$\mathcal{F}_{x_1} \neq \mathcal{F}_{x_2}$$
where $x_1,x_2$ lie above $x$ in the normalization. Thus $\mu^{\adj}(\mathcal{F}) \leq 2g+3$ and this rules out the remaining possibilities for a destabilizing subbundle, i.e. $\nu^* \mathcal{E}$ is semi-stable with respect to the adjusted slope and this finishes the proof of Theorem \ref{MT}.

\section{Tetragonal Curves}

\subsection{Introduction}

The goal of this section is to discuss the Harder-Narasimhan filtration for a tetragonal canonical curve. In this case the geometric Riemann-Roch Theorem implies that a tetragonal curve lies on a $3$-fold scroll $Q$ in $\PP^{g-1}$ and in \cite{AFO16} the authors showed that $N_{C/Q}$ is a destabilizing subbundle of $N_{C/\PP^{g-1}}$. Therefore we can ask what role the subbundle $N_{C/Q}$ plays in the Harder-Narasimhan filtration of $N_{C/\PP^{g-1}}$. We will focus almost entirely on the genus $6$ case, the outline of this section is as follows:

\begin{enumerate}
\item Introduce some background and state our main theorem which computes the HN-filtration of a general genus $6$ canonical curve.
\item Prove the main theorem by degenerating to a union of elliptic normal curves.
\item Show that $N_{Q/\PP^{5}} \rvert_C$ is semi-stable for a general genus $6$ tetragonal curve, which in particular implies that $N_{Q/\PP^{5}}$ is semi-stable. Furthermore we get that
$$N_{C/S} \subset N_{C/Q} \subset N_{C/\PP^{5}}$$
is a filtration of $N_{C/\PP^{5}}$ by semi-stable bundles. However, this filtration is not the HN-filtration because it does not satisfy the decreasing slope condition of Theorem \ref{T1}.
\end{enumerate}

For the general curve we can show that $N_{C/\PP^5}$ is unstable by using the fact that $C$ lies on a del Pezzo surface $S$. Showing the existence of such a surface $S$ starts with observing that every genus $6$ curve possesses a $g^2_6$. Indeed we compute
$$\rho(6,2,6)=6-3(6-6+2)=0$$
so that by the Brill-Noether existence Theorem $W^2_6(C) \neq \emptyset$. Alternatively in Chapter $5$ of \cite{ACGH} the authors use ad hoc methods to show $W^1_4(C) \neq \emptyset$. Then $W^2_6(C)$ is also nonempty because on a genus $6$ curve the residual of a $g^1_4$ is a $g^2_6$. Furthermore the exercises in Chapter $6$ of \cite{ACGH} show that a $g^2_6$ on a general genus $6$ curve maps $C$ birationally to a sextic plane curve with $4$ nodes. If we blowup $\PP^2$ at the four nodes we obtain a del Pezzo surface $S$ containing $C$. We will prove the following Theorem which computes the HN-filtration of $N_{C/\PP^5}$.

\begin{theorem}
Let $C$ be a general canonical curve of genus $6$. If $S \subset \PP^5$ is the del Pezzo surface containing $C$ then
$$0 \subset N_{C/S} \subset N_{C/\PP^5}$$
is the Harder-Narasimhan filtration of $N_{C/\PP^5}$.
\end{theorem}

Since $C$ is birational to a plane sextic with $4$ nodes we can compute the class of $C$ in $S$.
$$[C]=6H-2E_1-2E_2-2E_3-2E_4$$
By adjunction the canonical divisor of $C$ is the restriction of $D=3H-\sum E_i$ to $C$. The complete linear series $\lvert D \rvert$ embeds $S$ into $\PP^5$ and this embedding restricts to the canonical embedding on $C$. Observe that $\mu(N_{C/S})=[C]^2=20$ while
from the preliminary section we know that $\mu(N_{C/\PP^5})=35/2$ so that $N_{C/S}$ destabilizes $N_{C/\PP^5}$. We will use the theory of the adjusted slope from \cite{CLV22} to compute the HN-filtration by degenerating to a union of elliptic normal curves. The key that allows us to do this is a corollary of a result of Ein and Lazarsfeld \cite{EL92} which says that if $X \subset \PP^{d}$ is an elliptic normal curve then $N_{X/\PP^d}$ is semi-stable.

\subsection{Proof of the Theorem}

We need to show that $N_{S/\PP^5} \rvert_C$ is semi-stable for a general curve of genus $6$. Recall that $C$ has class $6H-2\sum E_i$ on the del Pezzo surface $S$. By Proposition \ref{P5} it suffices to show that $N_{S/\PP^5} \rvert_{X_1 \cup X_2}$ is semi-stable with respect to the adjusted slope where $X_1$ and $X_2$ both have class $3H-\sum E_i$. In other words the $X_j$ are the strict transform of cubics in $\PP^2$ passing through $p_1, \dots, p_4$ and as such they have genus $1$. Since $S$ is embedded in $\PP^5$ via the complete linear series $\lvert 3H-\sum E_i \rvert$ it follows that $X_1$ and $X_2$ are mapped into $\PP^5$ as hyperplane sections of $S$. If $X_j = \Lambda_j \cap S$ for $j=1,2$ where $\Lambda_j \cong \PP^4$ then by Lemma \ref{L2} we have
$$N_{S/\PP^5} \rvert_{X_j} \cong N_{X_j/\Lambda_j}$$
But $X_j \subset \Lambda_j$ is an elliptic normal curve so that $N_{X_j/\Lambda_j}$ is semi-stable by \cite{EL92}. Thus using Proposition \ref{P6} we conclude that $N_{S/\PP^5} \rvert_{X_1 \cup X_2}$ is semi-stable as desired.

\subsection{More on curves of genus 6}

Given a tetragonal genus $6$ canonical curve there is another filtration of $N_{C/\PP^5}$ by semi-stable bundles which involves $N_{C/Q}$. This is a three step filtration but it is not the HN-filtration since it does not satisfy the non-increasing slope condition required by the Harder-Narasimhan filtration. Recall that $C$ lies on a $3$-fold scroll $$Q=\PP(\O_{\PP^1}(1) \oplus \O_{\PP^1}(1) \oplus \O_{\PP^1}(1)) \cong \PP^1 \times \PP^2$$
and that the fibers of a $g^1_4$ on $C$ are given by lines passing through a single node $p_i$ or the conics passing through all four nodes $p_1, \dots, p_4$. Using this fact one can show that the del Pezzo surface $S$ containing $C$ is contained in the scroll $Q$. Thus there is a chain of inclusions
$$0 \subset N_{C/S} \subset N_{C/Q} \subset N_{C/\PP^5}$$
the claim is that for a general tetragonal canonical curve of genus $6$ this gives a filtration of $N_{C/\PP^5}$ by semi-stable bundles. Note that $N_{C/S}$ and $N_{S/Q} \rvert_C$ are lines bundles so both are semi-stable. Therefore the above will give a filtration if $N_{Q/\PP^5} \rvert_C$ is semi-stable for a general curve $C$. In order to show this we will need the following Lemma.

\begin{lemma}
If $Q \subset \PP^5$ is the image of the Segre embedding then for each line $L=\PP^1 \times \{p\}$ on $Q \cong \PP^1 \times \PP^2$ there is a quadric $Y_p$ in $\PP^5$ containing $Q$ whose singular locus is $L$. Furthermore if $\mathscr{I}_Q$ is the ideal sheaf of $Q$ then $$\PP(H^0(\mathscr{I}_Q(2)))=\{Y_p \mid p \in \PP^2\}$$.
\end{lemma}

\begin{proof}
We will realize the Segre threefold as the image of the embedding $\PP^1 \times \PP^2 \to \PP^5$ given by
$$([s:t],[x:y:z]) \mapsto [sx:sy:sz:tx:ty:tz]$$
Choose coordinates $z_i$ on $\PP^5$, then the ideal of the Segre threefold is generated by the equations
$$z_0z_4-z_1z_3=0$$
$$z_0z_5-z_2z_3=0$$
$$z_1z_5-z_2z_4=0$$
In particular the equation of any quadric containing $Q$ is a linear combination of the above equations, i.e. $H^0(\mathscr{I}_Q(2))=3$. The point $q_0=([1:0],[1:0:0]) \subset \PP^1 \times \PP^2 \subset \PP^5$ is contained in the line
$$L: z_0=z_1=z_3=z_4=0$$
and $L$ is the singular locus of the quadric
$$Y:z_0z_4-z_1z_3=0$$
If $q_1=(p_1,p_2)$ is any point of $\PP^1 \times \PP^2 \subset \PP^5$ then we can find an element $g \in \PGL(6,\C)$
such that $g(q_0)=q_1$ and $g$ fixes $Q$. Then the image $Y_{p_2}=g(Y)$ of $Y$ under $g$ is a quadric containing $Q$ which is singular along the line $\PP^1 \times \{p_2\} \subset \PP^1 \times \PP^2$. Thus we get a $2$-dimensional family of quadrics $\{Y_{p}\}_{p \in \PP^2}$ containing $Q$. But this must give all quadrics containing $Q$ since we already know the space of quadrics containing $Q$ is $2$-dimensional.
\end{proof}

Let $\alpha$ and $\beta$ be the pullbacks of hyperplane classes on $\PP^1$ and $\PP^2$ respectively. Then $Q \cong \PP^1 \times \PP^2$ is embedded in $\PP^5$ via $\lvert \alpha + \beta \rvert$. Since $C$ is tetragonal of degree $10$ in $\PP^5$ its class in $Q$ is
$$[C]=6\alpha\beta+4\beta^2$$
In particular if $\pi:\PP^1 \times \PP^2 \to \PP^2$ is the second projection then $\pi \rvert_C$ maps $C$ to a degree $6$ curve in $\PP^2$. Note that by the exercises in chapter $5$ of \cite{ACGH} for the general genus $6$ curve $\pi \rvert_C$ maps $C$ birationally to a plane sextic with four nodes $r_1, \dots, r_4 \in \PP^2$. For each $i$ the fiber $\phi^{-1}(r_i)$ consists of two points, i.e. the line $L_i=\PP^1 \times \{r_i\} \subset \PP^1 \times \PP^2$ intersects $C$ in two points $s_{i,1}, s_{i,2}$. For each $i$ if $\Bl_{L_i} \PP^5$ is the blowup of $\PP^5$ along $L_i$ we get an inclusion of vector bundles $\sigma: N_{Q/\hat{Y_{r_i}}} \rvert_C(s_{i,1}+s_{i,2}) \to N_{Q/\PP^5} \rvert_C$. A Chern class computation shows that $\coker(\sigma) \cong \O_{C}(2-s_{i,1}-s_{i,2})$, i.e. there is a short exact sequence
\[\begin{tikzcd}
	0 & {N_{Q/Y_{r_i}} \rvert_C(s_{i,1}+s_{i,2})} & {N_{Q/\PP^5} \rvert_C} & {\O_C(2-s_{i,1}-s_{i,2})} & 0
	\arrow[from=1-1, to=1-2]
	\arrow[from=1-2, to=1-3]
	\arrow[from=1-3, to=1-4]
	\arrow[from=1-4, to=1-5]
\end{tikzcd}\]
Note that $N_{Q/\PP^5} \rvert_C$ has degree $34$ while $\O_C(2-s_{i,1}-s_{i,2})$ and $N_{Q/Y_{r_i}} \rvert_C(s_{i,1}+s_{i,2})$ have degrees $16$ and $18$ respectively. If $N_{Q/\PP^5} \rvert_C$ has a destabilizing line subbundle $\mathcal{M}$ then either the induced map $\mathcal{M} \to \O_C(2-s_{i,1}-s_{i,2})$ is nonzero or $\mathcal{M} \subset N_{Q/Y_{r_i}} \rvert_C(s_{i,1}+s_{i,2})$. In the latter case $\mu(\mathcal{M}) \leq 16 < \mu(N_{Q/\PP^5} \rvert_C)$ and in the second case $\mu(\mathcal{M}) \leq 18$. If $\mu(\mathcal{M})=18$ then for each $i$ the map $\mathcal{M} \to \O_C(2-s_{i,1}-s_{i,2})$ would be an isomorphism, in particular this implies 
$$\O_C(2-s_{i,1}-s_{i,2}) \cong \O_C(2-s_{j,1}-s_{j,2})$$
for each $i,j$. But if we had such an isomorphism for $i \neq j$ then $C$ would be hyperelliptic which contradicts our assumption that $C$ is general. It follows that $\mu(\mathcal{M}) \leq 17$ for every line subbundle $\mathcal{M} \subset N_{Q/\PP^5} \rvert_C$, i.e. $N_{Q/\PP^5} \rvert_C$ is semistable. Furthermore $\mu(N_{Q/\PP^{5}} \rvert_C) = 17$ is odd implies that $N_{Q/\PP^5} \rvert_C$ is semi-stable iff it is stable.

We have now shown that
$$0 \subset N_{C/S} \subset N_{C/Q} \subset N_{C/\PP^5}$$
gives a filtration of $N_{C/\PP^5}$ by semi-stable bundles. We have already seen that $\mu(N_{C/S})=20$ and that $\mu(N_{C/Q})=18$ so that from the exact sequence
\[\begin{tikzcd}
	0 & {N_{C/S}} & {N_{C/Q}} & {N_{S/Q} \rvert_C} & 0
	\arrow[from=1-1, to=1-2]
	\arrow[from=1-2, to=1-3]
	\arrow[from=1-3, to=1-4]
	\arrow[from=1-4, to=1-5]
\end{tikzcd}\]
we conclude that $\mu(N_{S/Q} \rvert_C)=16$. On the other hand $\mu(N_{Q/\PP^5} \rvert_C)=17$ so that the slopes of the semi-stable factors in our three step filtration do not satisfy the decreasing slope condition required by the Harder-Narasimhan filtration.

\subsection{Final Thoughts}
Since we have shown that $N_{Q/\PP^{5}} \rvert_C$ is stable for some rational curve $C$ it follows that $N_{Q/\PP^{5}}$ must also be stable. Therefore the above argument might generalize and provide a strategy for determining the semi-stability of $N_{Q/\PP^{n}}$ when $Q$ is a rational normal scroll $Q \subset \PP^n$. We have also left unanswered several questions regarding the HN-filtration of $N_{C/\PP^{g-1}}$ for trigonal and tetragonal curves. For example we have not said anything about the HN-filtration of $N_{C/\PP^{g-1}}$ for tetragonal curves of genus $g \geq 7$. One question in this vein is if the subbundle $N_{C/Q}$ (where $Q$ is the threefold scroll containing $C$) plays a role in the HN-filtration when $g \geq 7$. Finally, recall that our argument in the trigonal case reduced to curves $C$ with Maroni invariant $n=0$ or $n=1$. We can therefore ask for the HN-filtration of curves with a larger maroni invariant. For example if $g$ is even then from \cite{ZSF00} we know that the locus of trigonal curves with Maroni invariant $n \geq 1$ forms a divisor in $\overline{\mathfrak{T}}_g$. The problem is to determine the HN-filtration for the canonical models of curves in this divisor.

\bibliographystyle{amsalpha}
\bibliography{refs}

\providecommand{\bysame}{\leavevmode\hbox to3em{\hrulefill}\thinspace}
\providecommand{\MR}{\relax\ifhmode\unskip\space\fi MR }
\providecommand{\MRhref}[2]{%
  \href{http://www.ams.org/mathscinet-getitem?mr=#1}{#2}
}
\providecommand{\href}[2]{#2}
\begin{thebibliography}{ACGH85}

\bibitem[ACGH85]{ACGH}
E.~Arbarello, M.~Cornalba, P.~A. Griffiths, and J.~Harris, \emph{Geometry of algebraic curves. {Volume} {I}}, Grundlehren Math. Wiss., vol. 267, Springer, Cham, 1985 (English).

\bibitem[AFO16]{AFO16}
Marian Aprodu, Gavril Farkas, and Angela Ortega, \emph{Restricted {L}azarsfeld-{M}ukai bundles and canonical curves}, Development of moduli theory---{K}yoto 2013, Adv. Stud. Pure Math., vol.~69, Math. Soc. Japan, [Tokyo], 2016, pp.~303--322. \MR{3586511}

\bibitem[ALY19]{ALY19}
Atanas Atanasov, Eric Larson, and David Yang, \emph{Interpolation for normal bundles of general curves}, Mem. Am. Math. Soc., vol. 1234, Providence, RI: American Mathematical Society (AMS), 2019 (English).

\bibitem[CLV22]{CLV22}
Izzet Coskun, Eric Larson, and Isabel Vogt, \emph{Stability of normal bundles of space curves}, Algebra Number Theory \textbf{16} (2022), no.~4, 919--953. \MR{4467126}

\bibitem[CLV23]{CLV23}
Izzet Coskun, Eric Larson, and Isabel Vogt, \emph{The normal bundle of a general canonical curve of genus at least 7 is semistable}, 2023.

\bibitem[CR19]{CR19}
Izzet Coskun and Eric Riedl, \emph{Normal bundles of rational curves on complete intersections}, Commun. Contemp. Math. \textbf{21} (2019), no.~2, 1850011, 29. \MR{3918047}

\bibitem[EH16]{3264}
David Eisenbud and Joe Harris, \emph{3264 and all that. {A} second course in algebraic geometry}, Cambridge: Cambridge University Press, 2016 (English).

\bibitem[EL92]{EL92}
Lawrence Ein and Robert Lazarsfeld, \emph{Stability and restrictions of {P}icard bundles, with an application to the normal bundles of elliptic curves}, Complex projective geometry ({T}rieste, 1989/{B}ergen, 1989), London Math. Soc. Lecture Note Ser., vol. 179, Cambridge Univ. Press, Cambridge, 1992, pp.~149--156. \MR{1201380}

\bibitem[GH78]{GH}
Phillip Griffiths and Joseph Harris, \emph{Principles of algebraic geometry}, Pure and {Applied} {Mathematics}. {A} {Wiley}-{Interscience} {Publication}. {New} {York} etc.: {John} {Wiley} \& {Sons}. {XII}, 1978.

\bibitem[Har77]{Hart}
Robin Hartshorne, \emph{Algebraic geometry}, Grad. Texts Math., vol.~52, Springer, Cham, 1977 (English).

\bibitem[Lar21]{HL20}
Hannah~K. Larson, \emph{Refined {Brill}-{Noether} theory for all trigonal curves}, Eur. J. Math. \textbf{7} (2021), no.~4, 1524--1536 (English).

\bibitem[LP97]{LePotier}
Joseph Le~Potier, \emph{Lectures on vector bundles}, Camb. Stud. Adv. Math., vol.~54, Cambridge: Cambridge University Press, 1997 (English).

\bibitem[Sch91]{FS91}
Frank-Olaf Schreyer, \emph{A standard basis approach to syzygies of canonical curves}, J. Reine Angew. Math. \textbf{421} (1991), 83--123 (English).

\bibitem[SF00]{ZSF00}
Zvezdelina~E. Stankova-Frenkel, \emph{Moduli of trigonal curves}, J. Algebr. Geom. \textbf{9} (2000), no.~4, 607--662 (English).

\bibitem[Vog18]{V18}
Isabel Vogt, \emph{Interpolation for {B}rill-{N}oether space curves}, Manuscripta Math. \textbf{156} (2018), no.~1-2, 137--147. \MR{3783570}

\end{thebibliography}
\nocite{*}

\end{document}